\documentclass{amsart}
\usepackage{amssymb}
\usepackage{mathrsfs}
\usepackage{amsmath}
\usepackage{hyperref}
\usepackage{amssymb}
\usepackage[english]{babel}

\usepackage{color}

\newtheorem{theorem}{Theorem}[section]
\newtheorem{corollary}[theorem]{Corollary}
\newtheorem{lemma}[theorem]{Lemma}
\newtheorem{proposition}[theorem]{Proposition}

\theoremstyle{definition}

\theoremstyle{remark}
\newtheorem{remark}[theorem]{Remark}

\newcommand{\be}{\begin{equation}}
\newcommand{\ee}{\end{equation}}
\numberwithin{equation}{section}

\newcommand{\R}{\mathbb R}
\newcommand{\N}{\mathbb N}
\newcommand{\Z}{\mathbb Z}
\newcommand{\T}{\mathbb T}

\newcommand{\co}{\colon}
\newcommand{\ep}{\varepsilon}
\newcommand{\pd}{\partial}
\newcommand{\const}{\operatorname{const}}

\newcommand{\bbp}{\mathbf{\bar p}}
\newcommand{\bbq}{\mathbf{\bar q}}
\newcommand{\bbc}{\mathbf{\bar c}}
\newcommand{\bbz}{\mathbf{\bar 0}}

\newcommand{\Vol}{\operatorname{Vol}}

\title{Flexibility of sections of nearly integrable Hamiltonian systems}
\author{Dmitri Burago, Dong Chen and Sergei Ivanov}

\address{Dmitri Burago: Pennsylvania State University, Department of Mathematics, University Park, PA 16802, USA}
\email{burago@math.psu.edu}
\address{Dong Chen: Pennsylvania State University, Department of Mathematics, University Park, PA 16802, USA}
\email{dxc360@psu.edu}
\address{Sergei Ivanov: St.Petersburg Department of Steklov Mathematical Institute, Russian Academy of Sciences, Fontanka 27, St.Petersburg 191023, Russia,  \&\newline
Saint Petersburg State University, 7/9 Universitetskaya emb., St. Petersburg 199034, Russia}
\email{svivanov@pdmi.ras.ru}
\date{}



\thanks{The first author was partially supported
by NSF grant DMS-1205597. The second author was partially supported by Dmitri Burago's Department research fund 42844-1001. The third author was partially supported by
RFBR grant  20-01-00070}

\begin{document}

\begin{abstract}
Given any symplectomorphism on $D^{2n} (n\geq 1)$ which is $C^{\infty}$ close to the identity, and any completely integrable Hamiltonian system $\Phi^t_H$ in the proper dimension,  we construct a $C^{\infty}$ perturbation of $H$ such that the resulting Hamiltonian flow contains a ``local
Poincar\'{e} section'' that ``realizes'' the symplectomorphism.  As a (motivating) 
application, we show that there are arbitrarily small perturbations of any completely integrable Hamiltonian system which are entropy non-expansive (and, in particular, exhibit hyperbolic behavior on a set of positive measure).  We use \textcolor{black}{some results in Berger-Turaev \cite{BT2}},
though in higher dimensions we could simply apply a construction from \cite{BI}.  
\end{abstract}

\maketitle

\section{Introduction}

We presume that a potential reader interested in this paper is familiar with such notions 
as symplectic manifolds and Lagrangian subspaces, Hamiltonian vector fields and flows, 
Poisson 
and Lie brackets, first integrals, Lyapunov exponents, metric and topological entropy, 
integrable systems, and has a basic idea of the main concept of KAM Theory. We
refresh some of these notions below, to set up notations, and  briefly discuss a few more 
technical aspects of KAM Theory. Open questions are at the end of the paper.

We work on  a  symplectic manifold $(\Omega^{2n},\omega_0)$ with $n\geq 2$. A Hamiltonian 
system (flow) with $n$ degrees of freedom is called \textit{completely integrable} if it enjoys $n$ 
algebraically independent first integrals which pair-wisely Poisson commute.

In a loose formulation, the main goal of this paper  is as follows. We are given a symplectic 
transformation of a ball and a completely integrable flow (in appropriate dimension). We assume 
that the transformation is close to the identity. We show that the Hamiltonian admits $C^\infty$ 
perturbations which are arbitrarily $C^\infty$ small and such that each perturbed flow contains,
in a natural sense,  a Poincar\'{e} section on which the return map is, after a proper rescaling and iterations, the given transformation \textcolor{black}{(the precise statement is a bit technical,  see Proposition \ref{perturb all levels} for more details).  Moreover, the section is taken near some periodic orbit of the unperturbed flow.  When the unperturbed Hamiltonian flow is the geodesic flow on the standard 
$\mathbb S^n$ ($n\geq 4$),  a ``dual lens map technique" has been recently developed  and used  in \cite{BI} to construct desired perturbations of Finsler metrics.  The ideas grew from Inverse and Boundary Rigidity Problems. 
The proof of our main result,  Proposition \ref{perturb all levels},  is based on a far-extended dual lens map technique.  In  some sense, in this paper the dual lens map argument 
is applied to Lagrangian submanifolds in a symplectic manifold rather than to geodesics in a Finsler manifold.}

\textcolor{black}{We begin with explaining an application that motivated this work. } According 
to the Liouville-Arnold Theorem (a precise statement can be found in \cite{KH}, p.227), except 
for a  zero measure set, 
the phase space of a completely integrable system 
with compact common level sets of the integrals
is foliated by 
invariant tori, called the {\it Liouville Tori}. The motion on each of these tori is conjugate to a linear flow on a standard torus. These invariant tori are in fact common level sets of the angle variables in the action-angle coordinates.

If one perturbs the Hamiltonian function of a completely integrable system, the resulting 
Hamiltonian flow is called \textit{nearly integrable}. 
Once the unperturbed system is nondegenerate 
in a suitable sense,
 the celebrated Kolmogorov-Arnold-Moser(KAM) Theorem \cite{Ar1}\cite{K}\cite{M} shows 
 that under a $C^{\infty}$ perturbation (though the smoothness can be lower depending on the degree of freedom), a large measure of invariant tori (called \textit{KAM tori}) survive and the dynamics 
 on these tori is still quasi-periodic. ``A large measure" means that the measure of the 
 tori which do not survive goes to zero as the size of the perturbation diminishes (the concrete
 estimates are of no importance for us here).  
 The tori which survive have rotation vectors whose directions are ``sufficiently irrational" (a certain degree 
 of being {\it Diophantine}, the precise condition is a bit technical and of no importance 
 for this paper). 
 
 The dynamic outside KAM tori draws a lot of attention.  Arnold \cite{Ar2} (in a number of papers followed by papers by Douady \cite{Do} and others) gave examples
of what is now known as the Arnold diffusion:  there may be trajectories asymptotic to one invariant torus at one end and then asymptotic to another torus on the other end.  Furthermore,  there might be trajectories which spend a lot of time near one torus, then leave and spend even longer time very close to another
one and so on. This sort of hyperbolicity is, however, very slow. We know this by the double-exponential
estimates on the transition time due to Nekhoroshev \cite{Ne}.   

An interesting but relatively easy 
(by modern standards, though quite important at its time) question is whether 
topological entropy could become positive. This means the presence of some hyperbolic 
dynamics there.  
Newhouse \cite{N1} proved that a $C^2$ generic Hamiltonian flow contains a hyperbolic set
(a horseshoe), hence the flow has positive topological entropy.  The Riemannian geodesic flows versions of this result were later established by Knieper and 
Weiss \cite{KW} for surfaces and Contreras \cite{C} for higher dimensions.

Positivity of topological entropy is nowadays not so exciting: it can (and often does) 
live on a set of zero measure. To get positive topological entropy, it suffices to find one Poincar\'{e}-Smale 
horseshoe (even of zero measure).  Little was known about \textit{metric entropy}, i.e. the measure theoretic 
entropy with respect to the Liouville measure on a level set (or to the symplectic volume on the entire space). Positive metric entropy implies positive 
topological entropy, but not vice versa \cite{BT1}. Despite the strong interest in nearly integrable 
Hamiltonian systems, what was lacking is understanding whether these systems may admit positive metric entropy. 

In this paper we give a positive answer to this question by constructing specific perturbations near any 
Liouville torus:

\begin{theorem}\label{thm2}
Let $\Phi^t_H$ be a completely integrable Hamiltonian flow on 
a symplectic manifold $\Omega=(\Omega^{2n},\omega)$ with $n\geq 2$. 
For any Liouville torus $\mathcal T\subset\Omega$,
one can find a $C^{\infty}$-small perturbation $\widetilde{H}$ of $H$ 
such that the resulting Hamiltonian flow $\Phi^t_{\widetilde{H}}$ 
has positive metric entropy.
Furthermore, such perturbation can be made in an arbitrarily small neighborhood
of $\mathcal T$ and such that the flow is entropy non-expansive. If $\Omega$ is compact,
there are perturbations with positive metric entropy with respect to $\omega^n$ on the whole $\Omega$.
\end{theorem}

{\it Remark 1}.  \textcolor{black}{Here a flow $\Phi^t$ is called \textit{entropy non-expansive} if the positive metric entropy can be generated in an arbitrarily small tubular neighborhood of an orbit. This issue attracts a lot of interest, see for instance \cite{B}\cite{N2}.  In particular, the first author \cite{Bu}  introduced this notion in 1988 being in mathematical isolation in the former Soviet Union.
 This situation is a bit counter-intuitive since hyperbolic dynamics tends to expand and occupy all space. \textcolor{black}{In fact, when Bowen defined entropy expansiveness in \cite{B}, he confessed he knew no example of entropy non-expansive diﬀeomorphism of a compact manifold.}  In our situation, however, the positive metric entropy can be obtained even  near any periodic orbit. }

{\it Remark 2}. 
The theorem holds for Hamiltonian (Lagrangian) flows that are geodesic flows on the co-tangent bundle of
Finsler manifolds. Furthermore, the perturbation can be made in the class of Finsler metrics. We leave to the
reader checking this. For the Riemannian metrics, however, this remains an open problem.  

{\it Remark 3}. Theorem \ref{thm2} almost answers perhaps the most intriguing question in the KAM theory.
The question, which is probably due to A. Kolmogorov, asks if arbitrarily small perturbations of a completely integrable system
may result in dynamics with positive metric entropy. Our construction, as an additional bonus, gives examples
which are entropy non-expansive, thus resolving another well-known problem. We say, however, 
``{\it almost answers}''
 since the construction is rather special. We do not know what happens for generic 
perturbations, and this probably is a quite difficult and important question.\\

By the dual lens map technique,  one can make a $C^{\infty}$ small Lagrangian perturbation of  the geodesic flow on the standard $\mathbb S^n (n\geq 4)$  to get positive (though extremely small due to \cite{N2}) metric entropy and the resulting flow is entropy non-expansive \cite{BI}. Together with the Maupertuis principle,  the second author managed to obtain positive metric entropy by perturbing the geodesic flow on an Euclidean $\mathbb{T}^n$ ($n\geq 3$)  \cite{Ch}. Unlike the case of spheres, the geodesic flows on flat tori are KAM-nondegenerate. 
Therefore in view of  KAM theory, the construction in \cite{Ch} is an improvement of that in \cite{BI}. With this result we know that in some region in the complement of KAM tori, the dynamics of a nearly integrable Hamiltonian flow can be quite stochastic on a set of positive measure. On the other hand, unlike the construction in \cite{BI}, the perturbed flows in \cite{Ch} are entropy expansive.

There are drastic differences between examples in \cite{BI}  \cite{Ch} and the one in this paper. 
Papers  \cite{BI} and  \cite{Ch} only deal with very specific systems: the geodesic flows of the standard metrics on 
$\mathbb S^n$ and $\mathbb T^n$;
\textcolor{black}{the original flow on $\mathbb S^n$
is KAM-degenerate
(and periodic) and the perturbed flow is entropy non-expansive, while those on $\T^n$ are KAM-nondegenerate and entropy expansive,  respectively.} In contrast to that, here
we can work with any completely integrable system, regardless of whether
it is KAM-nondegenerate or not,
and the perturbed flow is entropy non-expansive. 

Also, thanks to the  results in Berger-Turaev \cite{BT2}, our result applies to $n=2$.  In this case,  the 2-dimensional KAM tori separate the 3-dimensional energy level thus 
no Arnol'd diffusion is possible in such systems.  The existence of positive metric entropy between these tori is a little bit surprising. Before \cite{BT2},  we perturbed the identity map by symplectically embedding the geodesic flow in the cotangent bundle of a negatively curved surface into Euclidean space (see \textcolor{black}{\cite[Lemma 5.1]{BI}}), which inevitably required larger dimension.  One of the advantages of our approach,
however, is that we have more flexibility of perturbations than those in  \cite{BT2}. Of course, 
there is an obvious infinite-dimensional space of perturbations obtained by conjugating any one by
a symplectomorphism; here we can, however, make {\it essential} changes by varying the metric of the negatively curved surface we employ. 

The above improvements run into serious difficulties and require new techniques and ideas. The first challenge is how to generate positive metric entropy by perturbing the Poincar\'{e} map $R$. The result in \cite{BT2} allows us to perturb $R$ to get positive metric entropy on one level set, but it is not enough to guarantee positive metric entropy on neighboring level sets since the metric entropy may vanish under tiny perturbation. To overcome this obstruction, we build up very special symplectic coordinates on each level set and then apply the Morse-Bott Lemma to make sure the perturbations on different level sets behave in the same way. A more serious challenge is to realize the perturbed return map $\tilde{R}$ as the Poincar\'{e} map of some Hamiltonian $\tilde{H}$ that is $C^{\infty}$ close to  $H$. 
If $\tilde{R}$ is the time-one map of some ﬂow, one could construct $\tilde{H}$ using the suspension of the generating vector ﬁeld (see e.g. \cite{CHPZ21}).  \textcolor{black}{However, neither the map in Berger-Turaev \cite{BT2},  nor the $\tilde{R}$ in our construction, is generated by a flow.  Our key result, 
Proposition \ref{perturb all levels},  uses Lagrangian submanifolds to perturb a Hamiltonian flow so that a given (symplectic) diffeomorphism becomes the Poincar\'{e}  return map of a section transverse to a periodic trajectory near this trajectory, which is the major novelty in this paper}.

In a nutshell, we have three steps: first, we create a ``periodic spot"(a disc consisting of points with the same period, see \cite{GT10}\cite{GT17}\cite{T10} for related results of interest) in the cross-section; second, we insert a positive-entropy symplectomorphism of the disc there; and third, we extend this to a perturbation of the Hamiltonian.
\textcolor{black}{The first step} is more difficult in higher dimensions and we have
to use  a  seemingly non-trivial construction. 

The paper is organized in the following way: in Section 4, we transform the problem of perturbing the Hamiltonian 
to perturbation of the Poincar\'{e} map.  In Section 5, we  show how to 
get positive metric entropy in a small invariant set by perturbing a family of symplectic maps (including those 
we get in Section 4) and give a proof of Theorem \ref{thm2}. A more detailed plan of proof can be found in Section \ref{plan}.

\section{Acknowledgments}
We are grateful to Leonid Polterovich for his valuable help for boosting our understanding of the technique related to Lagrangian submanifolds. We also thank Vadim Kaloshin, late Anatole Katok, Federico Rodriguez Hertz,  Yakov Sinai,  Pierre Berger, and Dmitry Turaev for useful discussions.  We would like to express our gratitude to the anonymous referee for thoroughly reading our paper, and for numerous very useful remarks, corrections, and suggestions on the improvement of the paper. 

\section{Preliminaries}

\subsection{Hamiltonian flows}
Let $(\Omega^{2n}, \omega_0)$ be a $2n$-dimensional 
symplectic manifold and $H$ a smooth function on $T^*M$. 
The \textit{Hamiltonian vector 
field $X_H$} is defined as the unique solution to the equation
$$\omega_0(X_H, V)=dH(V)$$
for any smooth vector field $V$ on $\Omega$. $X_H$ is well-defined due to non-degeneracy of $\omega_0$. 
The flow $\Phi_H^t$ on $\Omega$ defined by $\dot{\Phi}_H^t=X_H$ is called \textit{the Hamiltonian flow on 
$\Omega$ with Hamiltonian $H$}. One can easily verify that $\Phi_H^t$ preserves $\omega_0$ and hence the 
volume form $\omega^n$. 
Any Hamiltonian flow is locally integrable. 
To be more specific, we have the following generalization of Darboux's theorem 
\cite[Chapter~I, Theorem~17.2]{LM}:

\begin{theorem}[Carath\'{e}odory-Jacobi-Lie]\label{CJL}
Let $(\Omega^{2n}, \omega_0)$ be a symplectic manifold. 
Let a family $p_1,...,p_k$ of $k$ differentiable functions ($k\leq n$), 
which are pairwise Poisson commutative and algebraic independent, 
be defined in the neighborhood $V$ of a point $x\in \Omega$. 
Then there exists $2n-k$ other functions $p_{k+1},...,p_{n},q_{1},...,q_{n}$ 
defined in an open neighborhood $U$ of $x$ in $V$ such that in $U$ we have
$$
\omega_0=\sum_{i=1}^n dq_{i}\wedge dp_i.
$$
\end{theorem}

\begin{corollary}\label{cor1}
For any point $x\in\Omega$ and any Hamiltonian function $H$, one can find an open neighborhood 
$U$ of $x$ and symplectic coordinates $(\mathbf q, \mathbf p)$ in $U$ such that 
$H|_U=p_n$.
\end{corollary}

\subsection{Sections and Poincar\'e maps}
\label{sec:poincare}

Throughout the paper $\Omega = (\Omega^{2n},\omega)$
denotes a symplectic manifold, $n\ge 2$,
$H\co\Omega \to \R$ a smooth Hamiltonian,
$X=X_H$ the Hamiltonian vector field of $H$,
and $\{\Phi_H^t\}_{t\in\R}$ the corresponding Hamiltonian flow.

A \textit{section} of the flow $\{\Phi_H^t\}$ is a $(2n-1)$-dimensional
smooth submanifold $\Sigma\subset\Omega$ transverse
to the trajectories of the flow. The transversality means
that $X_H$ is nowhere tangent to~$\Sigma$.
Note that this implies in particular that $X_H$ does not vanish on~$\Sigma$.
A section $\Sigma$ determines
the \textit{Poincar\'e return map} which sends a point $x\in\Sigma$
to the first intersection point
of the trajectory $\{\Phi_H^t(x)\}_{t>0}$ with $\Sigma$.
This is a partially defined map from $\Sigma$ to itself.

We need to consider a more general situation: given two sections
$\Sigma_0$ and $\Sigma_1$ of $\{\Phi_H^t\}$,
the associated \textit{Poincar\'e map} is a partially defined
map $R_H\co \Sigma_0\to\Sigma_1$ defined as follows:
for $x\in\Sigma_0$, $R_H(x)$ is the first intersection point
of the trajectory $\{\Phi_H^t(x)\}_{t>0}$ with $\Sigma_1$.
(If the trajectory does not intersect $\Sigma_1$,
then $R_H(x)$ is undefined).
We denote this map by $R_{H,\Sigma_0,\Sigma_1}$ or by $R_H$
when $\Sigma_0$ and $\Sigma_1$ are clear from context.

Since $\Sigma_0$ and $\Sigma_1$ are transverse to the
trajectories, $R_H$ is defined on an open subset of $\Sigma_0$
and it is a diffeomorphism from this subset to its image in $\Sigma_1$.
In this paper we always choose sections $\Sigma_0$ and $\Sigma_1$ so that
$R_{H,\Sigma_0,\Sigma_1}$ is a diffeomorphism between $\Sigma_0$ and $\Sigma_1$.
This is achieved by replacing $\Sigma_0$ and $\Sigma_1$ by suitable small
neighborhoods of some $x\in\Sigma_0$ and $R_H(x)\in\Sigma_1$.

The Hamiltonian induces a number of structures on sections.
Here is a list of structures and their properties that we
need in this paper.
For a detailed exposition of relations between
flows and their sections, see \cite[\S6.1]{Petersen}.

\subsubsection{Induced measure on sections}
Since the flow $\Phi_H^t$ preserves the canonical symplectic volume on~$\Omega$,
it naturally induces a measure $\Vol_{\Sigma}$ on a section $\Sigma$
as follows: for a Borel measurable $A\subset\Sigma$,
$$
 \Vol_{\Sigma}(A) = \Vol_\Omega \{ \Phi_H^t(x) : x\in A, \ t\in[0,1] \}
$$
where $\Vol_\Omega$ in the right-hand side is the symplectic volume counted with multiplicity.

One easily sees that Poincar\'e maps preserve the induced measure on sections.
Furthermore, from Abramov's formula \cite{Ab2}
one sees that the positivity of metric entropy of a Poincar\'e return map
implies that of the flow:

\begin{proposition}\label{p:entropy via return map}
Let $\Sigma$ be a section such that the Poincar\'e return map
$R_{H,\Sigma,\Sigma}$ is a diffeomorphism and it has positive
metric entropy. Then the flow $\{\Phi_H^t\}$ has positive metric entropy.
\qed
\end{proposition}

\subsubsection{Slicing sections by level sets}
For a section $\Sigma$ and $h\in\R$ we denote by $\Sigma^h$
the $h$-level set of $H|_{\Sigma}$:
$$
 \Sigma^h = \{ x\in \Sigma : H(x) = h \} .
$$
Since $\Sigma$ is transverse to the flow, the differential of $H|_\Sigma$
does not vanish. Hence $\Sigma^h$ is a smooth $(2n-2)$-dimensional submanifold.
Furthermore, one easily checks the following:
\begin{enumerate}
\item
The restriction of the symplectic form $\omega$ to $\Sigma^h$
is non-degenerate. Thus $\Sigma^h$ is a symplectic manifold.
\item
Let $\Sigma_0$, $\Sigma_1$ be two sections
such that the Poincar\'e map
$R_H\co\Sigma_0\to\Sigma_1$ is a diffeomorphism. 
Since the flow preserves $H$, $R_H$ sends $\Sigma_0^h$ to $\Sigma_1^h$.
We denote by $R_H^h$ the restriction $R_H|_{\Sigma_0^h}$.
One easily sees that $R_H^h$ preserves
the symplectic form, hence it is a symplectomorphism between $\Sigma_0^h$
and~$\Sigma_1^h$.
\item
Let $\Vol_{\Sigma^h}$ denote the $(2n-2)$-dimensional symplectic
volume on $\Sigma^h$. A straightforward computation shows that
the $(2n-1)$-dimensional volume $\Vol_\Sigma$ (see above) is
determined by the family $\{\Vol_{\Sigma^h}\}_{h\in\R}$ of symplectic
volumes of level sets:
\be\label{e:sliced volume}
 \Vol_\Sigma(A) = \int_{\R} \Vol_{\Sigma^h}(A\cap\Sigma^h) \,dh
\ee
\end{enumerate}

The identity \eqref{e:sliced volume} and Abramov-Rokhlin
entropy formula \cite{AbR} 
imply that it suffices to obtain positive metric entropy for the
Poincar\'e return map on the slices $\Sigma^h$.
Namely the following holds.

\begin{proposition} \label{p: entropy via slices}
Let $\Sigma$ be a section such that the Poincar\'e return map
$R_H=R_{H,\Sigma,\Sigma}$ is a 
self-diffeomorphism of $\Sigma$.
Suppose that there is a set $\Lambda\subset\R$ of positive Lebesgue measure
such that for every $h\in\Lambda$, the symplectomorphism
$R_H^h\co\Sigma^h\to\Sigma^h$ has positive metric entropy.
Then $R_H$ has positive metric entropy.
\end{proposition}

\begin{remark}\label{r:locality}
The above structures depend on both $\Sigma$ and $H$.
The measure $\Vol_\Sigma$ and the level sets $\Sigma^h$
are determined by the restriction of $H$ to an arbitrary
neighborhood of $\Sigma$.
In the proof of Theorem \ref{thm2} we fix suitable sections
$\Sigma_0$ and $\Sigma_1$ and perturb $H$ in a small region
between them.
The perturbation does not change $H$
near $\Sigma_0\cup\Sigma_1$, hence
the volume forms $\Vol_{\Sigma_i}$
on the sections and the splitting of $\Sigma_i$
into symplectic submanifolds $\Sigma_i^h$
remain the same.

However the Poincar\'e map can be affected by
 such perturbations of $H$.
Our plan is to perturb $H$ in such a way that the resulting
Poincar\'e return map on a suitable section $\Sigma$
has an invariant open set where it satisfies the assumptions
of Proposition~\ref{p: entropy via slices}.
\end{remark}

\subsection{Plan of the proof}\label{plan}
Let $\Omega$, $H$, $\mathcal T$ be as in Theorem \ref{thm2}.
First, by a small perturbation of $H$ preserving integrability
and Liouville tori, we make the flow on $\mathcal T$
nonvanishing and periodic. 
Then pick a point $y_0\in\mathcal T$ and choose a small section
$\Sigma$ through $y_0$ such that $y_0$ is a fixed point of
the Poincar\'e return map $R_H=R_{H,\Sigma,\Sigma}$.
Let $\Sigma_0$ be a small neighborhood of $y_0$ in $\Sigma$
such that $R_H|_{\Sigma_0}$ is a diffeomorphism onto its
image $\Sigma_1:=R_H(\Sigma_0)$.
The remaining perturbations of $H$ occur within a tiny
neighborhood of a point $x_0$ lying on the trajectory of $y_0$.
This guarantees that the Poincar\'e map 
$R_{\widetilde H}=R_{\widetilde H,\Sigma_0,\Sigma_1}$
is still a diffeomorphism between $\Sigma_0$ and $\Sigma_1$.

The perturbed Hamiltonian $\widetilde H$ should satisfy
the following conditions:
the Poincar\'e map $R_{\widetilde H}\co\Sigma_0\to\Sigma_1$
has an invariant open set $U\ni y_0$,
and the restriction of $R_{\widetilde H}$ to this invariant set
has a positive metric entropy.
Then, by Proposition \ref{p:entropy via return map} applied
to $U$ in place of $\Sigma$, the flow
$\{\Phi_{\widetilde H}^t\}$ has positive metric entropy.

The construction of $\widetilde H$ is divided into two parts.
The first part, summarized in Lemma \ref{l:perturb section},
is a construction of a perturbed Poincar\'e map
$\widetilde R\co\Sigma_0\to\Sigma_1$ with the
properties desired from the Poincar\'e map $R_{\widetilde H}$.
The second part, described in Section \ref{sec:R to H},
is a construction of a perturbed Hamiltonian $\widetilde H$
which realizes the given $\widetilde R$ as its Poincar\'e map:
$\widetilde R=R_{\widetilde H}$.

In order to be realizable as a Poincar\'e map, the diffeomorphism $\widetilde R$
has to satisfy the natural requirements:
it should map the slices $\Sigma_0^h$ to the respective slices $\Sigma_1^h$,
and it should preserve the symplectic form on the slices.
In fact, in Section \ref{sec:R to H} we show that any sufficiently small
compactly supported perturbation of $R_H$ satisfying these requirements
is realizable as a Poincar\'e map of some perturbed Hamiltonian,
see Proposition \ref{perturb all levels}.

\section{Hamiltonian perturbations with prescribed Poincar\'e maps}
\label{sec:R to H}

In this section we fulfill the last step of the above plan.
We use the notation introduced in Section \ref{sec:poincare}:
$\Omega=(\Omega^{2n},\omega)$ is a symplectic manifold, $n\ge 2$,
$H\co\Omega\to\R$ is a Hamiltonian and $\{\Phi_H^t\}$ is the
corresponding flow,
$\Sigma_0$ and $\Sigma_1$ are sections such that
the Poincar\'e map $R_H\co\Sigma_0\to\Sigma_1$
is a diffeomorphism.
Let $y_0\in\Sigma_0$ and let $x_0$ be a point
on the trajectory $\{\Phi_H^t(y_0)\}$ between $\Sigma_0$ and $\Sigma_1$.

Let $\widetilde R$ be a perturbation of $R_H$ with the same
properties as $R_H$, namely
\begin{enumerate}
 \item $\widetilde R\co\Sigma_0\to\Sigma_1$ is a diffeomorphism;
 \item $\widetilde R$ preserves $H$, that is, $H\circ\widetilde R = H$ on $\Sigma_0$.
 Equivalently, $\widetilde R(\Sigma_0^h)=\Sigma_1^h$ for every $h\in\R$;
 \item the restriction of $\widetilde R$ to each $\Sigma_0^h$
preserves the symplectic form.
\end{enumerate}
We also assume that $\widetilde R$ is $C^\infty$-close to $R_H$
and they coincide outside a small neighborhood of our base point $y_0$. 
Our goal is to realize $\widetilde R$ as a Poincar\'e map of some perturbed
Hamiltonian $\widetilde H$.
Moreover $\widetilde H$ can be chosen $C^\infty$-close to $H$ and such that
$\widetilde H-H$ is supported in a small neighborhood of~$x_0$.
More precisely, we prove the following.

\begin{proposition} \label{perturb all levels}
Let $\Omega$, $H$, $\Sigma_0$, $\Sigma_1$, $y_0$ and $x_0$
be as above.
Then for every neighborhood $U$ of $x_0$ in $\Omega$
there exists a neighborhood $V$ of $y_0$ in $\Sigma_0$
such that,
for every neighborhood $\mathcal H$ of $H$ in $C^\infty(\Omega,\R)$
there exists a neighborhood $\mathcal R$ of $R_H$ in $C^\infty(\Sigma_0,\Sigma_1)$
such that the following holds.

For every $\widetilde R\in \mathcal R$ satisfying $(1)$--$(3)$ above
and such that $\widetilde R=R_H$ outside $V$, there exists $\widetilde H\in \mathcal H$
such that $\widetilde H=H$ outside $U$,
and $\widetilde R = R_{\widetilde H}$ where $R_{\widetilde H}\co\Sigma_0\to\Sigma_1$
is the Poincar\'e map induced by $\widetilde H$.
\end{proposition}

In the sequel we assume that the neighborhood $U$
in Proposition \ref{perturb all levels}
is so small that $\overline U\cap(\Sigma_0\cup\Sigma_1)=\emptyset$
where $\overline U$ denotes the closure of~$U$.
This guarantees that
the Hamiltonian remains the same in a neighborhood of $\Sigma_0\cup\Sigma_1$
and therefore the induced structures on $\Sigma_0$ and $\Sigma_1$ are preserved
by the perturbation, see Remark~\ref{r:locality}.

In order to prove Proposition \ref{perturb all levels}, we first prove the following variant 
where we realize $\widetilde R$ 
as a Poincar\'e map only on one level set $H^{-1}(h)$. Here we denote by $R^h_H$ the restriction of $R_H$ on $\Sigma^h_0$.

\begin{proposition} \label{perturb one level}
Let $\Omega$, $H$, $\Sigma_0$, $\Sigma_1$, $y_0$ and $x_0$
be as above, and let $h=H(x_0)$.
Then for every neighborhood $U$ of $x_0$ in $\Omega$
there exists a neighborhood $V^h$ of $y_0$ in $\Sigma^h_0$
such that,
for every neighborhood $\mathcal H$ of $H$ in $C^\infty(\Omega,\R)$
there exists a neighborhood $\mathcal R^h$ of $R_H^h$ in $C^\infty(\Sigma^h_0,\Sigma^h_1)$
such that the following holds.

For every symplectic $\widetilde R^h\in \mathcal R^h$
such that $\widetilde R^h=R^h_H$ outside $ V^h$,
there exists $\widetilde H\in \mathcal H$
such that $\widetilde H=H$ on $H^{-1}(h)\setminus U$
and $\widetilde R^h = R^h_{\widetilde H}$.
\end{proposition}

\subsection{Proof of Propositions \ref{perturb one level}}

The proof of Proposition
\ref{perturb one level} 
is divided into a number of steps.

\medskip
\textbf{Step 1. Localization.}
In this step we show that it suffices to prove the proposition
in the canonical case where
\begin{itemize}
 \item $\Omega=\R^{2n} = \{ (\mathbf{q},\mathbf{p}) : \mathbf{q},\mathbf{p}\in\R^n \} $,
 with the standard symplectic structure $\omega=\sum_{i=1}^{n} dq_i\wedge dp_i$. 
 \item $H(\mathbf{q},\mathbf{p}) = p_n$.
 \item $x_0$ is the origin of $\R^{2n}$.
 \item $\Sigma_0= \{ (\mathbf{q},\mathbf{p}) : q_n=-1 \} $ and
$\Sigma_1= \{ (\mathbf{q},\mathbf{p}) : q_n=1 \} $.
\end{itemize}

Indeed,
our assumptions imply that $dH(x_0)\ne 0$.
By adding a constant to $H$ we may assume that $H(x_0)=0$.
By Theorem \ref{CJL} there exist a neighborhood $U_0$ of $x_0$ 
and a symplectic coordinate system
$(\mathbf{q},\mathbf{p})$, $\textbf{q}=(q_1,\dots,q_n)$, $\textbf{p}=(p_1,\dots,p_n)$,
in $U_0$ such that $p_n=H|_{U_0}$ and $\mathbf{p}(x_0)=\mathbf{q}(x_0)=0$.
In these coordinates we have $X_H = \frac{\partial}{\partial q_n}$.

Let $\ep_0>0$ be so small that the coordinate cube 
$$
 Q_0 := \{ (\mathbf{q},\mathbf{p}) : |q_i|\le \ep_0 \text{ and } |p_i|\le \ep_0 \text{ for all $i$} \}
$$
is contained in both $U_0$ and the union of the trajectories
between $\Sigma_0$ and $\Sigma_1$.
Let $\Sigma_-$ and $\Sigma_+$ be the opposite open faces of $Q$ defined by
\be\label{e:Sigma_pm}
 \Sigma_\pm =  \{ (\mathbf{q},\mathbf{p}) : q_n=\pm\ep_0, \ |q_i|<\ep_0 \text{ for all $i<n$}, 
  \ |p_i|<\ep_0 \text{ for all $i$}\} .
\ee
It suffices to prove the propositions for $\Sigma_-$ and $\Sigma_+$
in place of $\Sigma_0$ and $\Sigma_1$.

Indeed, the assumptions on $\ep_0$ imply that there are diffeomorphic Poincar\'e maps
$R_-=R_{H,\Sigma_0',\Sigma_-}$ and 
$R_+=R_{H,\Sigma_+,\Sigma_1'}$ where
$\Sigma_0'$ and $\Sigma_1'$ are suitable neighborhoods
of $y_0$ and $R_H(y_0)$ in $\Sigma_0$ and $\Sigma_1$, resp.
In the statements of the Propositions \ref{perturb all levels}
and \ref{perturb one level}
we may assume that $U\subset Q_0$ and require that $V\subset\Sigma_0'$
(resp.\ $V^h\subset\Sigma_0'$).
Then a perturbation $\widetilde H$ of $H$ does not
change the Poincar\'e map outside $\Sigma_0'$,
and it induces a Poincar\'e map
$\widetilde R\co\Sigma_0'\to\Sigma_1'$ if and only if it induces
a a Poincar\'e map $R_+^{-1}\circ\widetilde R\circ R_-^{-1}$ from $\Sigma_-$ to $\Sigma_+$.

Thus we may replace $\Sigma_0$ and $\Sigma_1$ with $\Sigma_-$ and $\Sigma_+$
and assume that $U\subset Q_0$. Using the coordinates to identify $Q_0$
with a cube in $\R^{2n}$ where $\R^{2n}$ is equipped 
with the standard symplectic structure and the standard Cartesian coordinates
$(\mathbf{q},\mathbf{p})=(q_1,\dots,q_n,p_1,\dots,p_n)$.
The hypersurfaces $\Sigma_\pm$ are now subsets of the affine hyperplanes
$\{q_n=\pm\ep_0\}$. By applying the symplectic transformation 
$(\mathbf{q},\mathbf{p})\mapsto (\ep_0^{-1}\mathbf{q},\mathbf{p}\cdot\ep_0)$
and multiplying $H$ by a constant, we make $\Sigma_\pm$ subsets of the
hyperplanes $\{q_n=\pm 1\}$, while $H$ is still the coordinate function $p_n$.
Now we may extend the structures to the whole $\R^{2n}$ and 
reduce the propositions to the canonical case described above.

Throughout the rest of the proof we work (without loss of generality) in this canonical setting.
Recall that the hypersurfaces $\Sigma_i$, $i=0,1$,
are foliated by level sets $\Sigma_i^h$ of~$H$.
In our standardized setting we have $\Sigma_i^h = \{ (\mathbf{q},\textbf{p}) \in\Sigma_i : p_n=h \} $,
so $\Sigma_i^h$ is an $(2n-2)$-dimensional affine subspace.

\medskip
\textbf{Step 2.}
For each $\mathbf{\widehat p}= (\widehat p_1,\dots,\widehat p_n)\in\R^n$,
define
\be \label{e:section slices}
 A_{\mathbf{\widehat p}} = \{ (\textbf{q},\textbf{p})\in \Sigma_0 : \mathbf{p}=\mathbf{\widehat p} \} .
\ee
Each set $A_{\mathbf{\widehat p}}$ is an $(n-1)$-dimensional affine subspace
contained in $\Sigma_0^h$ for $h=\widehat p_n$.
Moreover  $A_{\mathbf{\widehat p}}$ is a Lagrangian submanifold of $\Sigma_0^h$. 
We denote by 
$$\R^n_h=\{(\widehat p_1,\dots,\widehat p_n): \widehat p_n=h\}.$$
A map $\widetilde R^h$ satisfying the requirements of Proposition \ref{perturb one level}
maps the partition $\{A_{\mathbf{\widehat p}}\}_{\mathbf{\widehat p}\in\R_h^n}$
of $\Sigma^h_0$ to a partition of $\Sigma^h_1$ into Lagrangian submanifolds
$\widetilde R(A_{\mathbf{\widehat p}})$.
%
%
The next lemma shows that $\widetilde R^h$ is uniquely determined by
the resulting partition of $\Sigma^h_1$.

\begin{lemma}\label{lh:split}
Let $R^h_1,R^h_2\co\Sigma^h_0\to\Sigma^h_1$ be symplectomorphisms such that $R^h_1=R^h_2$ outside a compact subset of~$\Sigma^h_0$.
Suppose that
$R^h_1(A_{\mathbf{\widehat p}})=R^h_2(A_{\mathbf{\widehat p}})$
for every $\mathbf{\widehat p}\in \R^n_h$.
Then $R^h_1=R^h_2$.
\end{lemma}

\begin{proof}
Let $f=(R^h_2)^{-1}\circ R^h_1$.
The map $f$ is a symplectomorphism from $\Sigma^h_0$ to itself and it is the identity outside a compact set.
Let $e_1,\dots,e_n,e_{n+1},\dots,e_{2n}$ be the coordinate vectors corresponding to the coordinates $\mathbf q=(q_1,\dots,q_n)$, $\mathbf p=(p_1,\dots,p_n)$ in~$\R^{2n}$.
The affine space $\Sigma^h_0$ is naturally equipped with coordinates $(q_1,\dots,q_{n-1},p_1,\dots,p_{n-1})$.
The assumption that $R^h_1(A_{\mathbf{\widehat p}})=R^h_2(A_{\mathbf{\widehat p}})$
for all $\mathbf{\widehat p}\in \R^n_h$ implies that $f$ preserves the coordinates $p_1,\dots,p_{n-1}$.
Hence for every $(\mathbf q,\mathbf p)\in\Sigma_0$
the partial derivatives of $f$ at $(\mathbf q,\mathbf p)$ have the form
\be\label{e:df/dp}
 \frac{\pd f}{\pd p_j}(\mathbf q,\mathbf p) = e_{n+j} + v_j, \qquad j=1,\dots,n-1,
\ee
and
\be\label{e:df/dq}
 \frac{\pd f}{\pd q_i}(\mathbf q,\mathbf p) = w_i, \qquad i=1,\dots,n-1,
\ee
where $v_j, w_i$ belong to the linear span of $e_1,\dots,e_{n-1}$.

Since $f$ preserves the symplectic form $\omega$
on every slice $\{ p_n = \const \}$ of $\Sigma_0$,
the vectors $v_j$ and $w_i$ from \eqref{e:df/dp} and \eqref{e:df/dq}
satisfy
$$
  \omega(e_{n+j} + v_j, w_i) = \omega(e_{n+j},e_i) = \delta_{ij}
$$
for all $i, j\in\{1,\dots,n-1\}$,
where $\delta_{ij}$ is the Kronecker delta.
Since $v_j$ and $w_i$ are from the linear span of $e_1,\dots,e_{n-1}$,
we have $\omega(v_j,w_i)=0$.
Thus $\omega(e_{n+j},w_i)=\delta_{ij}$ for all $i,j$.
Hence $w_i=e_i$ for all~$i$.
Now \eqref{e:df/dq} takes the form
$$
 \frac{\pd f}{\pd q_i}(\mathbf q,\mathbf p) = e_i, \qquad i=1,\dots,n-1 .
$$
Hence the restriction of $f$ to every subset $\{\mathbf q = \const \}$
is a parallel translation. Since $f$ is the identity outside a compact set,
it follows that $f$ preserves the coordinates $q_1,\dots,q_{n-1}$ and is the identity everywhere.
Hence $R^h_1=R^h_2$.
\end{proof}


We may assume that the set $U$ where we are allowed to change the Hamiltonian
is a cube $(-\ep,\ep)^{2n}$ where $\ep\in(0,1)$.
We prove the statement of 
Proposition \ref{perturb one level} for $V^h\subset\Sigma^h_0$ defined as the projection of $U$ to $\Sigma^h_0$, namely
\be\label{e:Vh definition}
  V^h = (-\ep,\ep)^{n-1}\times\{-1\}\times (-\ep,\ep)^{n-1}\times\{h\} .
\ee

We may assume that the neighborhood $\mathcal H$ of $H$
(see the formulation of the proposition)
is so small that every $\widetilde H\in\mathcal H$
such that $\widetilde H=H$ on $H^{-1}(h)\setminus U$
induces a smooth bijective Poincar\'e map
$R^h_{\widetilde H}\co\Sigma^h_0\to\Sigma^h_1$.
Moreover, $R^h_{\widetilde H}=R^h_H$ outside $V^h$ for every such $\widetilde H$.

With Lemma \ref{lh:split}, 
Proposition \ref{perturb one level} boils down to the following statement:
Given a sufficiently small perturbation $\widetilde R^h$ of $R^h_H$ 
such that  $\widetilde R^h=R^h_H$ outside $V^h$,
we can construct a perturbed Hamiltonian $\widetilde H\in\mathcal H$ such that
$\widetilde H=H$ on $H^{-1}(h)\setminus U$ and
\be
  R^h_{\widetilde H}(A_{\mathbf{\widehat p}})= \widetilde R^h(A_{\mathbf{\widehat p}})
\quad
\text{for all $\mathbf{\widehat p}\in\R^n_h$}.
\ee

\medskip
{\bf Step 3.}
For each $\mathbf{\widehat p}\in\R^n$, define a Lagrangian affine subspace
$L_{\mathbf{\widehat p}}\subset\R^{2n}$ by
$$
L_{\mathbf{\widehat p}} = \{ (\mathbf q, \mathbf{\widehat p}) : q\in\R^n \} .
$$
The subspaces $L_{\mathbf{\widehat p}}$, where $\mathbf{\widehat p}$ ranges over $\R^{n}$,
form a foliation of $\R^{2n}$.
Our plan is to perturb the subfoliation $\{L_{\mathbf{\widehat p}}\}_{\mathbf{\widehat p}\in\R^n_h}$ and obtain another foliation by Lagrangian submanifolds $\{\widetilde L_{\mathbf{\widehat p}}\}_{\mathbf{\widehat p}\in\R^n_h}$
such that
\be\label{e:L-left}
 \widetilde L_{\mathbf{\widehat p}} \cap\Sigma^h_0 = A_{\mathbf{\widehat p}}
\ee
and 
\be\label{e:L-right}
\widetilde L_{\mathbf{\widehat p}} \cap\Sigma^h_1 = \widetilde R^h(A_{\mathbf{\widehat p}})
\ee
for all $\mathbf{\widehat p}\in\R_h^n$, and define the perturbed Hamiltonian $\widetilde H$
so that it is constant on each submanifold $\widetilde L_{\mathbf{\widehat p}}$. Since the flow (without fixed points) on a level set of a Hamiltonian is determined by this set up to a time change, once we fix the level set $\widetilde H^{-1}(h)$,  the resulting Poincar\'e map on $\Sigma^h_0$ is independent of other level sets.

The next lemma says that this construction solves our problem.
We say that a line segment $[x,y]\subset\R^{2n}$ is \textit{horizontal} if
it is parallel to the coordinate axis of the $q_n$-coordinate.

\begin{lemma}\label{l:leaves}
Let $\{\widetilde L_{\mathbf{\widehat p}}\}_{\mathbf{\widehat p}\in\R_h^n}$ be a foliation
 by Lagrangian submanifolds satisfying \eqref{e:L-left} and \eqref{e:L-right}.
Let $\widetilde H\co\R^{2n}\to\R$ be a smooth function such that
\be\label{e:H on fibers}
 \widetilde H|_{\widetilde L_{\mathbf{\widehat p}}} = h \quad\text{for all \ $\mathbf{\widehat p}\in\R_h^n$}
\ee
and suppose that $\widetilde H$ defines a smooth Poincar\'e map $R^h_{\widetilde H}\co\Sigma^h_0\to\Sigma^h_1$.

Suppose in addition that every horizontal segment intersecting $\Sigma^h_0\cup\Sigma^h_1$ 
but not intersecting $U=(-\ep,\ep)^{2n}$
is contained in one of the submanifolds $\widetilde L_{\mathbf{\widehat p}}$. Then
$R^h_{\widetilde H} = \widetilde R^h$
and
$ \widetilde H=H $
on $H^{-1}(h)\setminus U$.
\end{lemma}

\begin{proof}
The key implication of the assumption \eqref{e:H on fibers} is that
every submanifold $\widetilde L_{\mathbf{\widehat p}}$ is contained in
one level set of~$\widetilde H$.

First we show that $\widetilde H=H$ on $H^{-1}(h)\setminus U$. Recall that $H=p_n$
and $\widetilde R(\Sigma^h_0)=\Sigma^h_1$.
This and \eqref{e:L-left}, \eqref{e:L-right}, \eqref{e:H on fibers}
imply that $\widetilde H=H$ on $\Sigma^h_0\cup\Sigma^h_1$.
Since $U$ is a convex set lying between the hyperplanes $\Sigma_0$ and $\Sigma_1$,
every point $x\in H^{-1}(h)\setminus U$ can be connected to a point $y\in\Sigma^h_0\cup\Sigma^h_1$
by a horizontal segment not intersecting $U$. By the assumptions of the lemma
the segment $[x,y]$ is contained in one level set of $\widetilde H$,
and it is contained in a level set of $H$ since $H=p_n$.
Therefore $\widetilde H(x)=H(x)$ for all $x\in H^{-1}(h)\setminus U$.

Now we show that $R^h_{\widetilde H} = \widetilde R^h$.
Fix $\mathbf{\widehat p}\in\R_h^n$ and consider the leaf $\widetilde L_{\mathbf{\widehat p}}$
of our foliation.
By elementary linear algebra, the property that
$\widetilde L_{\mathbf{\widehat p}}$ is contained in
a level set of~$\widetilde H$ implies that the Hamiltonian vector field $X_{\widetilde H}$
is tangent to $\widetilde L_{\mathbf{\widehat p}}$.
Therefore $\widetilde L_{\mathbf{\widehat p}}$ is invariant under the flow $\Phi^t_{\widetilde H}$.
By \eqref{e:L-left} and \eqref{e:L-right} it follows that
\be\label{e:R(A)}
 R^h_{\widetilde H}(A_{\mathbf{\widehat p}}) = L_{\mathbf{\widehat p}} \cap \Sigma^h_1
 = \widetilde R^h(A_{\mathbf{\widehat p}}) .
\ee
The fact that $\widetilde H=H$ on $H^{-1}(h)\setminus U$ implies that $R^h_{\widetilde H}=R^h_H=\widetilde R^h$
outside a compact set.
Now with \eqref{e:R(A)} at hand we can apply Lemma \ref{lh:split} to $R^h_{\widetilde H}$ and $\widetilde R^h$
in place of $R^h_1$ and $R^h_2$ and conclude that $R^h_{\widetilde H}=\widetilde R^h$.
\end{proof}

It remains to construct a foliation $\{\widetilde L_{\mathbf{\widehat p}}\}$ satisfying Lemma \ref{l:leaves}
and such that the resulting Hamiltonian $\widetilde H$ is sufficiently close to $H$ in $C^\infty$.
This is achieved in the next two steps.

\medskip
{\bf Step 4.}
We begin with a construction of $\widetilde L_{\mathbf{\widehat p}}$ for a fixed $\mathbf{\widehat p}\in\R_h^n$.
Throughout this step, $\widehat p_i$ denotes a fixed real number
(the $i$th coordinate of $\mathbf{\widehat p}$)
and $p_i$, $q_i$ are still coordinate functions,
$i=1,\dots,n$.

We identify $\R^{2n}$ with the cotangent bundle $T^*\R^n$ using $q_i$'s as spatial coordinates
and $p_i$'s as coordinates in the fibers of the cotangent bundle.
We construct the desired leaf $\widetilde L_{\mathbf{\widehat p}}$ as a graph of a closed 1-form
$\widetilde\alpha=\widetilde\alpha_{\mathbf{\widehat p}}$ on $\R^n$.
Recall that a graph of a 1-form is a Lagrangian submanifold of the cotangent bundle
if and only if the 1-form is closed, see e.g. \cite{CdS}.

Define $W=(-\ep,\ep)^n$.
The last requirement of Lemma \ref{l:leaves} prescribes $\{\widetilde L_{\mathbf{\widehat p}}\}$
and hence $\widetilde\alpha$ outside $W$.
We cover $\R^{2n}$ by two open half-spaces $\{q_n<\ep\}$ and $\{q_n>-\ep\}$
and consider $\widetilde\alpha$ on these half-spaces with $W$ removed.
First consider a 1-form
$\alpha = \sum \widehat p_i\,dq_i$
(with constant coefficients) and define
the restriction of $\widetilde\alpha$ to $\{q_n<\ep\}\setminus W$ by
\be\label{e:tilde alpha left}
\widetilde\alpha=\alpha \qquad \text{on $\{q_n<\ep\}\setminus W$} .
\ee
The graph of $\alpha$ is the unperturbed leaf $L_{\mathbf{\widehat p}}$.
It satisfies \eqref{e:L-left} and consists of horizontal segments,
fulfilling the respective part of requirements of Lemma \ref{l:leaves}.

Now we define $\widetilde\alpha$ on the set $\{q_n>-\ep\}\setminus W$.
In fact,  $\widetilde\alpha$ on this set is uniquely determined 
by the map $\widetilde R$ and the requirement~\eqref{e:L-right}.
Consider the set $\widetilde A:=\widetilde R^h(A_{\mathbf{\widehat p}})$,
the desired intersection of the graph of $\widetilde\alpha$ with $\Sigma^h_1$.
Since $\widetilde R$ preserves $H$ and the symplectic form in the levels of $H$,
$\widetilde A$ is an $(n-1)$-dimensional Lagrangian submanifold of the affine subspace
$\Sigma^h_1$.
In the unperturbed case $\widetilde R^h=R^h_H$,  $\widetilde A$ 
is an affine subspace of $\Sigma^h_1$.
Hence, if $\widetilde R^h$ is sufficiently close to $R^h_H$ in $C^\infty$,
then $\widetilde A$ 
is a graph of a closed 1-form $\widetilde\beta$ 
defined on the hyperplane $\{q_n=1\}\subset\R^n$.
We define the restriction of $\widetilde\alpha$ to $\{q_n>-\ep\}\setminus W$
by 
\be\label{e:tilde alpha right}
\widetilde\alpha=\Pi^*\widetilde\beta + \widehat p_n\,dq_n  \qquad \text{on $\{q_n>-\ep\}\setminus W$} ,
\ee
where $\Pi$ is the orthogonal projection from $\R^n$ to the hyperplane $\{q_n=1\}$.
The graph of the 1-form defined by \eqref{e:tilde alpha right} consists of horizontal segments
and satisfies \eqref{e:L-right}.
Since $\widetilde R^h=R^h_H$ outside $V$ (see \eqref{e:Vh definition}),
the definitions \eqref{e:tilde alpha left} and \eqref{e:tilde alpha right}
agree on the common domain $\{-\ep<q_n<\ep\}\setminus W$.

Thus we have defined the desired 1-form $\widetilde\alpha$ on $\R^n\setminus W$.
Our goal is to extend $\widetilde\alpha$ to the whole $\R^n$.
We need the following lemma.

\begin{lemma}\label{l:exact}
$\widetilde\alpha$ defined above is exact on $\R^n\setminus W$.
\end{lemma}

\begin{proof}
The statement is trivial if $n>2$, since the 1-form is closed and 
the set $\R^n\setminus W$ is simply connected.

For $n=2$, it suffices to check that the integral of $\widetilde\alpha$
over any one cycle going around the hole $W=(-\ep,\ep)^2$, is zero.
We do this for the boundary of the square $[-1,1]^2$. Let $s$ be the side
$[-1,1]\times\{1\}$ of this square.
Since $\widetilde\alpha=\alpha$ on the remaining three sides of the square,
it suffices to verify that $\int_s\widetilde\alpha=\int_s\alpha$.
Each integral is the signed area between the graph of the respective 1-form
and the line $\{p_1=0\}$ in the plane $\Sigma^h_1$
with coordinates $(q_1,p_1)$. Since one graph is taken to the other
by a symplectomophism $\widetilde R^h\circ (R^h_H)^{-1}$ which is the identity
outside the small square $(-\ep,\ep)^2$, this signed area is preserved.
\end{proof}

The right-hand sides of \eqref{e:tilde alpha left} and \eqref{e:tilde alpha right}
are closed 1-forms defined on the entire $\R^n$.
Let $f\co\R^n\to\R$ and $g\co\R^n\to\R$ be their antiderivatives.
To ensure that $f$ and $g$ depend smoothly on the parameter $\mathbf{\widehat p}$,
we choose them so that $g(0,\dots,0,-1)=f(0,\dots,0,-1)=0$.
Observe $f=g$ outside the set $(-\ep,\ep)^{n-1}\times\R$
since \eqref{e:tilde alpha left} and \eqref{e:tilde alpha right} agree there.

We combine $f$ and $g$ using a suitable partition of unity as follows.
Fix a smooth function $\mu\co\R\to[0,1]$ such that
$\mu(t)=1$ for all $t<-1/2$ and $\mu(t)=0$ for all $t>1/2$
and define a smooth function $\widetilde f\co\R^n\to\R$ by
$$
 \widetilde f(\mathbf q) = \mu(q_n/\ep)\cdot f(\mathbf q) + (1-\mu(q_n/\ep))\cdot g(\mathbf q) .
$$
Now define
$$
 \widetilde\alpha = d\widetilde f
$$
everywhere on $\R^n$.
This definition agrees with \eqref{e:tilde alpha left} and \eqref{e:tilde alpha right}
on their respective domains since $f$ and $g$ agree on the set $\{-\ep<q_n<\ep\}\setminus W$.

This finishes the construction of the 1-form $\widetilde\alpha=\widetilde\alpha_{\mathbf{\widehat p}}$
for a fixed $\widehat p$. The graph $\widetilde L_{\mathbf{\widehat p}}$ of $\widetilde\alpha_{\mathbf{\widehat p}}$
is a Lagrangian submanifold satisfying the requirements of Lemma \ref{l:leaves}.
\medskip

{\bf Step 5.} It remains to show that $\widetilde{H}$ can be chosen to be $C^{\infty}$ close to $H$. In order to prove it we prove the family of Lagrangian submanifolds $\{\widetilde L_{\mathbf{\widehat p}}\}_{\mathbf{\widehat p}\in\R_h^n}$, constructed in the previous step, is a $C^{\infty}$ perturbation of the original foliation $\{L_{\mathbf{\widehat p}}\}_{\mathbf{\widehat p}\in\R_h^n}$.

Going through the constructions of Step~4 one sees that $\widetilde\alpha_{\mathbf{\widehat p}}$
depends smoothly on $\widehat p$.
Hence one can define a smooth map $F^h\co H^{-1}(h)\to\R^{2n}$ by
$$
 F^h(\mathbf q,\mathbf{\widehat p}) = (\mathbf q, \alpha_{\mathbf{\widehat p}}(\mathbf q)) ,
 \qquad \mathbf q\in\R^n,\mathbf{\widehat p}\in \R^n_h.
$$
This map takes each leaf of the original foliation $\{L_{\mathbf{\widehat p}}\}$ to
the corresponding Lagrangian submanifold $\widetilde L_{\mathbf{\widehat p}}$.
Note that $F^h$ is determined by $\widetilde R^h$ by means of explicit formulae
(involving some inverse functions) and in the unperturbed case 
(when $\widetilde R^h=R^h_H$) the resulting map $F^h$ is the identity.
Therefore $F$ is $C^\infty$-close
to the identity on any fixed compact set
as long as $\widetilde R$ is $C^\infty$-close to $R_H$. We may assume that the neighborhood $\mathcal R$ of $R^h_H$
from which $\widetilde R^h$ is chosen (see the formulation of Proposition \ref{perturb one level})
is so small that
\be\label{e:Fh close to id}
 \|F^h-\mbox{id}\|_{C^1([-1,1]^{2n-1}\times\{h\})} < \tfrac12
\ee
where the norm of the first derivative is understood as the operator norm.

The construction in Step 4 implies that $F$ is the identity on the set
$\{q_n<-\ep\}$ and $F^h-\mbox{id}$ is constant along any horizontal segment
not intersecting $U$. This implies that the norm estimate in \eqref{e:Fh close to id}
holds in $C^1(H^{-1}(h))$, and this norm estimate implies that
$F^h$ is a diffeomorphism from $H^{-1}(h)$ to its image.

Thus $\widetilde{H}$ can be chosen to be $C^{\infty}$ close to $H$ and we finish the proof of Proposition \ref{perturb one level}.\\

Proposition \ref{perturb one level} can be applied to prove the following fact which is known in folklore but for which
the authors could not find a reference.

\begin{proposition}\label{p:family}
Let $\varphi_0\co D^{2n}\to D^{2n}$, $n\ge 1$, be a symplectomorphism $C^\infty$-close to the identity
and coinciding with the identity near the boundary. Then there exists a smooth family
of symplectomorphisms $\{\varphi_t\}_{t\in[0,1]}$ of $D^{2n}$
fixing a neighborhood of the boundary
and such that $\varphi_t=\varphi_0$ for all $t\in[0, \frac13]$, $\varphi_t=\operatorname{id}$ for all $t\in[\frac23,1]$,
and the family $\{\varphi_t\}$ is $C^\infty$-close to the trivial family (of identity maps).
\end{proposition}
\begin{proof}

Consider $\Omega=\R^{2n+2}=(q_1,\dots,q_{n+1},p_1,\dots,p_{n+1})$
with the standard symplectic structure and Hamiltonian $H=p_{n+1}$.
et $\Sigma_0$ and $\Sigma_1$ be affine hyperplanes
defined by the equations $q_{n+1}=-1$ and $q_{n+1}=1$, resp.

We introduce notation $\bbp$ and $\bbq$ for the coordinate $n$-tuples
$(p_1,\dots,p_n)$ and $(q_1,\dots,q_n)$.
The Poincar\'e map $R_H\co\Sigma_0\to\Sigma_1$ is given by
$$
 R_H(\bbq,-1,\bbp,p_{n+1}) = (\bbq,1,\bbp,p_{n+1}).
$$
We define a perturbed map $\widetilde R^0\co\Sigma^0_0\to\Sigma^0_1$ by
$$
 \widetilde R^0(\bbq,-1,\bbp,0) = (\varphi_{0,q}(\bbq,\bbp), 1, \varphi_{0,p}(\bbq,\bbp), 0)
$$
where $\varphi_{0,q}$ and $\varphi_{0,p}$ are $q$- and $p$- coordinates of $\varphi_0$. By applying Proposition \ref{perturb one level} to $h=0$ we get $\widetilde H$ such that $\widetilde R^0=R^0_{\widetilde H}$. 
Let $\mathscr{H}$ be the Hamiltonian such that $\mathscr{H}^{-1}(0)=\widetilde H^{-1}(0)$ (hence $\widetilde R^0=R^0_{\widetilde H}=R^0_{\mathscr{H}}$), and $\mathscr{H}-p_{n+1}$ does not depend on $p_{n+1}$. Then $\partial\mathscr{H}/\partial p_{n+1}=1$ and the return map $R_{\mathscr{H}}$ is given by
$$
R_{\mathscr{H}}(\bbq,-1,\bbp,p_{n+1}) = (\varphi_{0,q}(\bbq,\bbp), 1, \varphi_{0,p}(\bbq,\bbp), p_{n+1}).
$$
Define a time-dependent Hamiltonian $H_t$ on $\R^{2n}$ as follows:
$$
  H_t(\bbq,\bbp) = \mathscr{H}(\bbq,t,\bbp,0) .
$$
Then the flow on $\R^{2n}$ generated by $H_t$ connects the identity to $\varphi_0$.
\end{proof}

\subsection{Proof of Proposition \ref{perturb all levels}}
We deduce Propositions \ref{perturb all levels} from Proposition \ref{perturb one level} by applying it to all $h\in\R$ and to the corresponding restrictions $\widetilde R|_{\Sigma_0^h}$ in place of $\widetilde R^h$. Let $\widetilde H^h$ denote the resulting perturbed Hamiltonian for a given $h$. Then the desired $\widetilde H$ in Proposition \ref{perturb all levels} is constructed from the family $\{\widetilde H^h\}_{h\in\R}$ in such a way that $\widetilde H^{-1}(h) = (\widetilde H^h)^{-1}(h)$ for every $h\in\R$. The resulting Poincar\'e map is the prescribed $\widetilde R$ since the Hamiltonian flow is determined by level sets up to a time change.

We have only to show that the family of Lagrangian submanifolds $\widetilde L_{\mathbf{\widehat p}}$,
$\widehat p\in\R^n$, constructed in the Step 4 in the proof of Proposition \ref{perturb one level}, forms a foliation of $\R^{2n}$. Similar to Step 5, one can define a smooth map $F\co\R^{2n}\to\R^{2n}$ by
$$
 F(\mathbf q,\mathbf{\widehat p}) = (\mathbf q, \alpha_{\mathbf{\widehat p}}(\mathbf q)) ,
 \qquad \mathbf q,\mathbf{\widehat p}\in\R^n.
$$
$F$ is $C^\infty$-close
to the identity on any fixed compact set
as long as $\widetilde R$ is $C^\infty$-close to $R_H$. Furthermore one can slightly adjust the argument in Step 5 to show that $F$ is a diffeomorphism from $\R^{2n}$ to itself. This finishes the proof of Proposition \ref{perturb all levels}.

\section{Proof of Theorem \ref{thm2}}

We begin with the following result by Berger-Turaev \cite{BT2}.
\begin{theorem}[\cite{BT2}]\label{BT}
For any $n\geq 1$, there is a $C^{\infty}$-small perturbation of the identity map 
$\operatorname{id}\co D^{2n}\rightarrow D^{2n}$ such that the resulting map is symplectic and coincides
with the identity map near the boundary and has positive metric entropy.
\end{theorem}

\begin{proof}
Theorem A in \cite{BT2} is proven for $n=1$. 
Though they did not mention in the theorem whether the perturbed map agrees 
with the original map near the boundary,
Corollary 5 in \cite{A2} (see also Corollary 4.8 in \cite{BT2}) guarantees 
that they can coincide near $\partial D$.

To extend the result to $n\ge 2$, one can do the following.
Let $\varphi\co D^2\to D^2$ be the perturbation of the identity constructed for $n=1$
and $\{\varphi_t\}$ the family of symplectomorphisms constructed in Proposition \ref{p:family}
for $\varphi_0=\varphi$.
Define a diffeomorphism $\Phi\co D^2\times D^{2n-2}\to D^2\times D^{2n-2}$ by
$$
 \Phi(x,y) = (\varphi_{|y|}(x), y), \qquad x\in D^2, \ y\in D^{2n-2} .
$$
One easily sees that $ \Phi$ is a symplectomorphism fixing the neighborhood of the boundary.
Since $ \Phi(x,y)=\varphi(x)$ for all $y$ with $|y|\le\frac13$ and $\varphi$ has positive metric entropy,
so does~$\Phi$.
\end{proof}

Let $\Omega=(\Omega^{2n},\omega)$ be a symplectic manifold,
$H\co\Omega\to\R$ a Hamiltonian such that the flow $\{\Phi_H^t\}$
is completely integrable, and $\mathcal T$ a Liouville torus and $x_0\in\mathcal T$.
By the Liouville-Arnold theorem, there exist action-angle coordinates
$(\mathbf q,\mathbf p)=(q_1,...,q_n, p_1,..,p_n)$ near $\mathcal T$.

These coordinates identify a neighborhood of $\mathcal T$ with
the product $\T^n\times D$ where $D\subset\R^n$ is a small $n$-dimensional disc
and $\T^n=\R^n/\Z^n$ is the standard $n$-torus. 
The coordinates $p_1,\dots,p_n$ parametrize $D$ and $q_1,\dots,q_n$
are the standard angle coordinates on $\T^n$ (taking values in $\R/\Z$).
The Hamiltonian $H$ depends only on $\mathbf p$-coordinates,
hence it can be regarded as a function on $D$.

The flow $\Phi_H^t$ in these coordinates is governed by the equations
\be\label{e:action angle flow}
\begin{cases}
  \dot q_i = A_i(\mathbf p) := \frac{\pd H}{\pd p_i}(\mathbf p) \\
  \dot p_i = 0 .
\end{cases}
\ee
Thus, along every trajectory the $\mathbf p$-coordinates are constant
and $q$-coordinates vary linearly with velocity $A_i(\mathbf p)$, $i=1,\dots,n$.
We may assume that $\mathbf p=\mathbf 0$ on $\mathcal T$
and $H(\mathbf 0)=0$.

By a small perturbation of the function $H=H(\mathbf p)$ near $\mathbf p=\mathbf 0$
we can satisfy the following conditions:
\begin{itemize}
 \item The flow on $\mathcal T$ is nonvanishing. (This means that at least one of the numbers
 $\frac{\pd H}{\pd p_i}(\mathbf 0)$ is nonzero).
 \item The flow on $\mathcal T$ is periodic. (It suffices to perturb $H$ so that
 all numbers $\frac{\pd H}{\pd p_i}(\mathbf 0)$ are rational).
 \item The system is KAM-nondegenerate at $\mathcal T$. In our notation this condition means
 that the Hessian of $H$ at $\mathbf p=0$ is nondegenegate.
\end{itemize}
We change the coordinates by an action of some matrix from $SL(n,\mathbb{Z})$ on $\mathbb{T}^n$ 
to assure that the (periodic) flow on $\mathcal T$ is the flow along the $q_n$-coordinate,
that is, $A_i(\mathbf 0)=0$ for $i=1,\dots,n-1$ and $A_n(\mathbf 0)>0$.

After having made these modifications we 
abuse notation and use the same letter $H$ for the modified Hamiltonian
and $\mathbf p$, $\mathbf q$ for the modified coordinates.
It suffices to prove the theorem for Hamiltonians and coordinates
satisfying the conditions above.

For $h\in\R$, denote
$$
D_h:=\{\mathbf p\in D : H(\mathbf p)=h\} .
$$
Replacing $D$ by a smaller neighborhood of $\mathbf 0$ if necessary,
we apply the implicit function theorem 
(using the fact that $\frac{\pd H}{\pd p_n}(\mathbf 0) = A_n(\mathbf 0) > 0$)
and obtain a smooth family $\{f_h\}_{h\in\R}$ of smooth functions $f_h\co\R^{n-1}\to\R$
such that
\be\label{e:implicit0}
  \mathbf p \in D_h \quad\iff\quad p_n=-f_h(p_1,\dots,p_{n-1})
\ee
for every $h\in\R$ and $\mathbf p\in D$.
The minus sign here is introduced to be canceled out later
in \eqref{e:implicit-derivative}.
We introduce notation $\bbp$ and $\bbq$ for the coordinate $(n-1)$-tuples
$(p_1,\dots,p_{n-1})$ and $(q_1,\dots,q_{n-1})$.
With this notation \eqref{e:implicit0} implies that 
\be\label{e:implicit}
  H(\bbp, -f_h(\bbp)) = h
\ee
for all $\bbp\in\R^{n-1}$ sufficiently close to the origin and $h\in\R$ sufficiently close to~0.
Differentiating \eqref{e:implicit} with respect to $p_i$ we obtain that
\be\label{e:implicit-derivative}
\frac{\pd f_h}{\pd p_i}(\bbp)=\frac{A_i}{A_n}(\bbp, f_h(\bbp)) , \qquad i=1,...,n-1.
\ee
Since $A_i(\mathbf 0)=0$ for $i<n$, the origin is a critical point of $f_h$.

Now we cut our invariant tubular neighborhood of $\mathcal T$ by a hypersurface
$$
\Sigma = \{(\mathbf q,\mathbf p) : q_n = 0\}
$$ and consider the resulting
Poincar\'e return map $R=R_H\co\Sigma\to\Sigma$.
The hypersurface $\Sigma$ is naturally identified with $\T^{n-1}\times D$
and parametrized by coordinates $(\bbq,\mathbf p)$
where $\bbq=(q_1,\dots,q_{n-1})$ and $\mathbf p=(p_1,\dots,p_n)$.
By \eqref{e:action angle flow}, $R$ is given by
\be\label{e:R formula}
  R(\bbq,\mathbf p) = \left(q_1+\tfrac{A_1}{A_n}(\mathbf p),\dots,q_{n-1}+\tfrac{A_{n-1}}{A_n}(\mathbf p), \mathbf p \right) .
\ee
Note that the origin of $\Sigma$ is a fixed point of $R$.

The next lemma is one of the key ingredients of the proof.

\begin{lemma}\label{l:perturb section}
There exists a diffeomorphism $\widetilde R\co\Sigma\to\Sigma$ arbitrarily
close to $R$ in $C^\infty$ and such that $\widetilde R=R$
outside an arbitrarily small neighborhood of the origin and the following
conditions are satisfied:
\begin{enumerate}
 \item For every $h\in\R$, 
 $\widetilde R$ maps the level set $\Sigma^h:=\{x\in\Sigma : H(x)=h\}$ to itself
 and preserves the symplectic form on this set.
 \item There is a small $\widetilde R$-invariant neighborhood of the origin and
 the restriction of $\widetilde R$ to this neighborhood has positive metric entropy.
 Moreover, $\widetilde R$ is entropy non-expansive.
\end{enumerate}
\end{lemma}

\begin{remark}
In order to speak about metric entropy of $\widetilde R$, we regard $\Sigma$
with the measure induced by the symplectic volume on $\Omega$ and the
original flow $\Phi_H^t$, see \eqref{e:sliced volume}.
One easily sees that any map $\widetilde R$ satisfying the first requirement
of the lemma preserves this measure.
\end{remark}

\begin{proof}
Our nondegeneracy assumption on the Hessian of $H(\mathbf p)$
implies that $\bbz\in\R^{n-1}$ is a nondegenerate critical point of~$f_0$.
Thus for all $h$ near $0$, the function $f_h$ has a nondegenerate
critical point $\bbc(h) = (c_1(h),\dots,c_{n-1}(h))$ depending smoothly on $h$
with $\bbc(0)=\bbz$.
 
First we fix $h\in\R$ sufficiently close to 0 and describe the construction
within~$\Sigma^h$.
By \eqref{e:implicit0}, the intersection of $\Sigma^h$ with a suitable neighborhood
of the origin is parametrized by a map 
$$
\Gamma_h \co O_q\times O_p \to \Sigma 
$$
given by
$$
 \Gamma_h(\bbq,\bbp) = (\bbq,\bbp, -f_h(\bbp))
$$
where $O_q$ and $O_p$ are certain neighborhoods of the origin in $\R^{n-1}$.
Here in the right-hand side we use the coordinates $(\bbq,\mathbf p)=(\bbq,\bbp,p_n)$
on $\Sigma$. 
Since $O_p$ comes from the implicit function theorem,
it can be chosen uniformly in $h$. Hence we may assume that $\bbc(h)\in O_p$.
Observe that $\Gamma_h$ preserves the symplectic form,
where $O_q\times O_p\subset\R^{2n-2}$ is equipped with the standard symplectic form 
$d\bbq\wedge d\bbp=\sum_{i=1}^{n-1} dq_i\wedge dp_i$.
Therefore the restriction of $R$ to the set $\Gamma_h(O_q\times O_p)\subset\Sigma^h$
is conjugate to a symplectic map $G_h\co O_q\times O_p\to\R^{2n-2}$,
\be\label{e:conjugation}
  G_h = \Pi\circ R\circ \Gamma_h
\ee
where $\Pi$ is the standard projection forgetting the last coordinate.
For brevity, we define $m=n-1$.

We are going to perturb $G_h$ so that the resulting map 
$\widetilde G_h\co O_q\times O_p\to\R^{2m}$ is still symplectic,
it coincides with $G_h$ outside a compact subset of $O_q\times O_p$,
has an invariant neighborhood of $(\bbz, \bbc(h))$ and positive metric entropy
in this neighborhood.

By \eqref{e:R formula} and \eqref{e:implicit-derivative}, $G_h$ can be written in the form
\be\label{e:G_h formula} 
  G_h(\bbq,\bbp) = (\bbq + \nabla f_h(\bbp), \bbp)
\ee
where $\nabla f_h$ is the Euclidean gradient of $f_h\co\R^m\to\R$.
Notice that $G_h$ is the time-1 map of the Hamiltonian flow 
$\Phi^t_{F_h}$ with the Hamiltonian $F_h$ given by
\be\label{e:F_h}
 F_h(\bbq,\bbp):=f_h(\bbp)
\ee
Our plan is to perturb $F_h$ and define $\widetilde G_h$
as the time-1 map of the flow defined by the perturbed Hamiltonian.

Since $\bbc(h)$ is a nondegenerate critical point of $f_h$, 
by Morse Lemma there exist a coordinate chart 
$\bar{\textbf{P}}=(P_1,\dots,P_m)$ in $O_p$
such that $\bar{\textbf{P}}$ vanishes at $\bbc(h)$ and 
\be\label{e:morse}
f_h=f_h(\bbc(h))+P_1^2+\cdots+P_k^2-P_{k+1}^2-\cdots-P_m^2
\ee
in a neighborhood of $\bbc(h)$.
We regard $P_1,\dots,P_m$ as functions on $O_q\times Q_p$
by setting $P_i(\bbq,\bbp)=P_i(\bbp)$.
Then \eqref{e:morse} is a formula for $F_h$ as well, cf.~\eqref{e:F_h}.
Since $P_1,\dots,P_m$ depend only on $\bbp$-coordinates,
they are Poisson commuting.
Hence, by Theorem \ref{CJL}, one can extend this collection
of functions to a symplectic coordinate system 
$(\bar{\mathbf Q},\bar{\mathbf P})=(Q_1,\dots,Q_m,P_1,\dots,P_m)$
in a neighborhood of the point $(\bbq,\bbp)=(\bbz,\bbc(h))$ in $O_q\times O_p$.
We may assume that $Q_1,...,Q_m$ vanish at $(\bbz, \bbc(h))$.

We perturb the Hamiltonian $F_h$ in a region
$U_{\delta}:=\{\bar{\mathbf P}^2<\delta,\bar{\mathbf Q}^2<\delta\}$ 
where $\delta$ is a sufficiently small positive number.
Let $\xi$ be a smooth function on $[0,1]$ with $\xi\equiv 1$ on $[0,\delta/2]$ and $\xi\equiv 0$ on $[\delta,1]$. 
For any $\ep>0$, define a perturbed Hamiltonian $F_{h,\ep,\delta}$ by
\be\label{e:F_h,ep,delta}
F_{h,\ep,\delta}
:=F_h + \ep\,\xi(\bar{\mathbf P}^2)\,\xi(\bar{\mathbf Q}^2)\,(Q_1^2+\cdots+Q_k^2-Q_{k+1}^2-\cdots-Q_m^2). 
\ee
Due to the formula \eqref{e:morse} for $F_h$, the Hamiltonian flow 
$\Phi^t_{F_{h,\ep,\delta}}$ 
within $U_{\delta/2}$ is governed by the following equations
in coordinates $(\bar{\mathbf Q},\bar{\mathbf P})$:
$$
\begin{cases}
 \dot Q_i = 2P_i , & i \le k, \\
 \dot P_i = -2\ep Q_i , & i \le k, \\
 \dot Q_i = -2P_i , & i > k, \\
 \dot P_i = 2\ep Q_i , & i > k .
\end{cases}
$$
This defines a periodic flow with period $\pi/\sqrt\ep$ and a fixed point at $\bar{\mathbf P}=\bar{\mathbf Q}=\bar{\mathbf 0}$.
Outside $U_\delta$, the flow $\Phi^t_{F_{h,\ep,\delta}}$ coincides with the original flow $\Phi^t_{F_{h}}$.
We assume that $\delta$ is so small that the trajectories of the flow $\Phi^t_{F_{h}}$ starting in $U_\delta$
stay within the domain of $(\bar{\mathbf Q},\bar{\mathbf P})$ for all $t\in [0,1]$.
Then the same property holds for the flow $\Phi^t_{F_{h,\ep,\delta}}$.

We choose $\ep$ so that $N:=\pi/\sqrt{\ep}$ is an integer.
Define
\be\label{e:G_h,ep,delta}
G_{h,\ep, \delta}:=\Phi^1_{F_{h,\ep, \delta}} ,
\ee
the time-1 map of the flow determined by the Hamiltonian $F_{h,\epsilon, \delta}$.
This map is defined on an open subset of $O_q\times O_p$ containing the closure
of $U_\delta$ (provided that $\delta$ is sufficiently small),
and it tends to $G_h$ in $C^\infty$ as $\ep\to 0$ (for each fixed $\delta$).
The disc $D_{\delta/2}$ is invariant under $G_{h,\ep, \delta}$ and the restriction
of $G_{h,\ep, \delta}$ to this disc is $N$-periodic.

Choose a closed disc $B\subset U_{\delta/2}$ such that the sets
$B, G_{h,\ep, \delta}(B),\dots, G_{h,\ep, \delta}^{N-1}(B)$ 
are disjoint.
This disc is just a sufficiently small ball centered at a non-fixed point
of $G_{h,\ep, \delta}$.
By Theorem \ref{BT}, there exist a symplectomorphism $\theta\co B\to B$
arbitrarily $C^\infty$-close to the identity fixing a neighborhood of the boundary
and having positive metric entropy.
We extend $\theta$ to the whole $\R^{2m}$ by the identity map outside $B$
and use the same letter $\theta$ for its extension to $\R^{2m}$.

Now define 
\be\label{e:tilde G_h}
\widetilde G_h = G_{h,\ep, \delta} \circ \theta .
\ee
This formula defines $\widetilde G_h$ in a neighborhood of the closure of $U_\delta$.
Outside $U_\delta$ this map coincides with $G_h$ and we extend it by $G_h$ to obtain
a map $\widetilde G_h\co O_q\times O_p\to\R^{2m}$.
By construction, $U_{\delta/2}$ is invariant under $\widetilde G_h$,
$B$ is invariant under $\widetilde G_h^N$, and
$
 (\widetilde G_h^N)|_B = \theta|_B 
$.
Therefore the restriction of $\widetilde G_h$ to $U_{\delta/2}$ has positive metric entropy.
By choosing $\ep$ sufficiently small and $\theta$ sufficiently close to the identity,
$\widetilde G_h$ can be made arbitrarily close to $G_h$ in the $C^\infty$ topology.

In order to make the perturbed map entropy non-expansive, the construction
can be modified as follows. Instead of working with one disc $B$, we choose a sequence of disks
$\{B_i\}$ tending to the origin, with diameters going to 0, and such that the sets
$G_{h,\ep,\delta}^k(B_i)$, $i\in\N$, $k=0,\dots,N-1$, are disjoint.
We perturb the identity map within each $B_i$ as in Theorem~\ref{BT}
so that the composition of these perturbations
is a $C^\infty$ map
$\theta\co\R^{2m}\to\R^{2m}$ which is close to the identity. 
Then the map $\widetilde G_h$ defined by \eqref{e:tilde G_h}
is entropy non-expansive.

By the conjugation inverse to \eqref{e:conjugation} we transform
$\widetilde G_h$ to a perturbation $\widetilde R_h$ of $R_h=R|_{\Sigma^h}$.
Namely
\be\label{e:tilde R_h}
 \widetilde R_h = \Gamma_h\circ\widetilde G_h \circ\Pi|_{\Sigma^h}
\ee
within the coordinate domain $\Gamma_h(O_q\times O_p)$ and
$\widetilde R_h$ coincides with $R_h$ outside this domain.
This finishes the description of the construction within one level set.

It remains to show that one can apply the construction simultaneously on all
level sets $\Sigma^h$, $h\in\R$, so that the union of the resulting maps
$\widetilde R_h$ is a diffeomorphism $\widetilde R\co\Sigma\to\Sigma$
satisfying the requirements of the lemma.

In order to do this, we first construct coordinates 
$(\bar{\mathbf Q},\bar{\mathbf P})=(\bar{\mathbf Q}^h,\bar{\mathbf P}^h)$
as above for all $h$ from a neighborhood of~0 so that they depend smoothly on~$h$.
The $\bar{\mathbf P}$-coordinates are constructed using the Morse Lemma.
In order to make sure that they depend smoothly on~$h$, one can apply
the Morse-Bott Lemma (a.k.a.\ the parametric Morse Lemma), see e.g.\ \cite[Theorem 2]{BH}.
More precisely, to obtain a smooth family $\{f_h\}$ of functions
satisfying \eqref{e:morse}, one applies the Morse-Bott Lemma to the function
$
 (\bbp, h) \mapsto f_h(\bbp) - f_h(\bbc(h))
$
defined in a neighborhood of~0 in $\R^{n-1}\times\R$.
The $\bar{\mathbf Q}$-coordinates are constructed from $\bar{\mathbf P}$-coordinates
by means of Theorem~\ref{CJL}. 
By analyzing the proof of Theorem \ref{CJL} in \cite{LM}, 
one can see that this construction boils down to explicit formulae
involving algebraic computations and solutions of O.D.E.s,
hence it can be made smooth in~$h$ in a suitable neighborhood.

Having constructed the $(\bar{\mathbf Q}^h,\bar{\mathbf P}^h)$-coordinates
for all $h\in(-h_0,h_0)$,
we define $F_{h,\ep,\delta}$ by \eqref{e:F_h,ep,delta}
using a small fixed $\delta$ and $\ep=\ep(h)$ such that
$\ep(h)$ is a small constant for $|h|<h_0/3$ and $\ep(h)=0$ for $|h|>2h_0/3$.
Then define $\widetilde G_h$ by   \eqref{e:G_h,ep,delta} and \eqref{e:tilde G_h}
using $\theta=\theta_h$ depending on $h$ as follows: $\{\theta_h\}$ is a smooth family
$C^\infty$-close to a constant one, $\theta_h$ is a fixed map $\theta$ as above for $|h|<h_0/3$,
and $\theta_h=\operatorname{id}$ for $|h|>2h_0/3$.
The existence of such a family is guaranteed by Proposition \ref{p:family}.
Finally, define $\widetilde R_h$ by \eqref{e:tilde R_h}.
The union of maps $\widetilde R_h$ forms a self-diffeomorphism
of the set $\{x\in\Sigma : H(x)\in(-h_0,h_0)\}$.
By construction, this diffeomorphism coincides with $R$ 
on the set of $x$ such that $|H(x)|\in(2h_0/3,h_0)$.
We extend it to a diffeomorphism $\widetilde R\co\Sigma\to\Sigma$
by setting $\widetilde R=R$ on the remaining part of~$\Sigma$.

The resulting map $\widetilde R$ has an invariant neighborhood
$\{ \mathbf P^2 < \delta/2, \mathbf Q^2 < \delta/2, |H|<h_0/3 \}$.
Since the maps $\widetilde R_h$, $h\in(-h_0/3,h_0/3)$,
have the coordinate representation and they have positive metric entropy
and are entropy non-expansive, the restriction of $\widetilde R$ 
to the above neighborhood has
a positive metric entropy.
\end{proof}

Now Theorem \ref{thm2} follows from Lemma \ref{l:perturb section}
and Proposition \ref{perturb all levels}.

\section{Some open problems} \label{open}

Here we briefly discuss a few open problems, some of them are mentioned above. 
\begin{enumerate}
\item In case of the geodesic flow on a Riemannian manifold, we do not know how to make the 
perturbation Riemannian. This seems to be quite an intriguing problem.

\item How large entropy can be generated depending on the size of perturbation (any estimate would certainly involve some characteristics of the unperturbed system)? Probably some (very non-sharp) lower bounds
can be obtained by a careful analysis of the proof. As for the upper bounds, we suspect they
should be double-exponential like Nekhoroshev estimates.

\item  Our construction is very specific and non-generic. What about a generic perturbation?
\end{enumerate}

\end{document}